\documentclass[11pt,reqno]{amsart}

\usepackage[hmargin=4cm,vmargin=4cm]{geometry}
\usepackage{subcaption}

\usepackage{amscd,amssymb}
\usepackage{tikz}
\usepackage{tikz-cd}

\newcommand\1{\lower 9pt\hbox{\underbar{}}}
\numberwithin{equation}{section}

\newtheorem {theorem}[equation]                   {Theorem}
\newtheorem {lemma}[equation]           {Lemma}

\theoremstyle{definition}
\newtheorem {definition}[equation]{Definition}
\newtheorem {remark}[equation]          {Remark}

\usepackage{hyperref}

\newcommand{\pr} {\smallskip\noindent{\bf Proof\,\,}}

\usetikzlibrary{arrows}

\DeclareMathOperator{\dive}{div}

\def\R{\mathbb{R}}

\newcommand{\pp}[2]{\frac{\partial#1}{\partial#2}}

\begin{document}

\title[Morse functions and contact convex surfaces]{Morse functions and contact convex surfaces}

\author{Robert Cardona}\address{Robert Cardona\\IRMA, UMR 7501, 7 rue Ren\'e-Descartes, 67084 Strasbourg Cedex}\email{robert.cardona@math.unistra.fr}

\author{C\'edric Oms}\address{C\'edric Oms \\UMPA, ENS de Lyon Site Monod,	 	 	 
46 All\'ee d'Italie,	 
69007 Lyon}\email{cedric.oms@ens-lyon.fr}

\thanks{R. Cardona thanks the LabEx IRMIA and the Universit\'e de Strasbourg for support. C. Oms is partially supported by the ANR grant ``Cosy" (ANR-21-CE40-0002). Both authors are partially supported by the ANR grant ``CoSyDy" (ANR-CE40-0014) and by the AEI grant PID2019-103849GB-I00 of MCIN/ AEI /10.13039/501100011033.}

\begin{abstract}
 Let $f$ be a Morse function on a closed surface $\Sigma$ such that zero is a regular value and such that $f$ admits neither positive minima nor negative maxima. In this expository note, we show that $\Sigma\times \mathbb{R}$ admits an $\mathbb{R}$-invariant contact form $\alpha=fdt+\beta$ whose characteristic foliation along the zero section is (negative) weakly gradient-like with respect to $f$. The proof is self-contained and gives explicit constructions of any $\mathbb{R}$-invariant contact structure in $\Sigma \times \mathbb{R}$, up to isotopy. As an application, we give an alternative geometric proof of the homotopy classification of $\mathbb{R}$-invariant contact structures in terms of their dividing set.
\end{abstract}
\maketitle
	
\section{Introduction}

The study of embedded surfaces in contact $3$-manifolds has a rich history, tracing back to the works of Bennequin \cite{Ben}, Eliashberg \cite{El} and Giroux \cite{G1}.  The contact structure \emph{prints} a $1$-dimensional singular foliation on the surface $\Sigma$, called the \emph{characteristic foliation}, that has singularities at the points where the tangent space of the surface agrees with the contact structure.  The divergence of the vector field is non-vanishing at the singular points of the foliation. The characteristic foliation is particularly well understood when the surface is convex - meaning that there exists a contact vector field transverse to the surface. Equivalently, this means that there exists a contact form defining the contact structure that is invariant with respect to the contact vector field and therefore called $\R$-invariant. The seminal work of Emmanuel Giroux \cite{G1,G2} establishes the foundations of convex surface theory, and shows for instance that an embedded surface in a contact three-manifolds is generically convex. Even more surprisingly, when $\Sigma$ is convex, the contact structure around it is determined by an isotopy class of embedded separating curves on $\Sigma$. These curves are called the \emph{dividing set}, which is defined as the set of points where the contact vector field is tangent to the contact structure. 

In this note, we are interested in the converse problem: Given a closed orientable surface $\Sigma$ and a set of dividing curves $\Gamma \subset \Sigma$, does there exist an $\mathbb{R}$-invariant contact form in $\Sigma\times \mathbb{R}$ whose dividing set is given by $\Gamma$? We show that one can construct such a contact form and that its characteristic foliation can be ``prescribed": It can be chosen to be weakly gradient-like for any Morse function whose zero regular level set is $\Gamma$ and that has neither a minimum where the function is positive nor a maximum where the function is negative. 

\begin{theorem}\label{thm:1}
 Let $f:\Sigma \to \R$ be a Morse function on an orientable closed surface $\Sigma$ such that zero is a regular value, and such that it does not admit positive minima nor negative maxima. Then $\Sigma\times \mathbb{R}$ admits an $\mathbb{R}$-invariant contact form $\alpha=fdt+\beta$ whose characteristic foliation along the zero section is (negatively) weakly gradient-like for $f$. 
\end{theorem}

In particular, the dividing set is given by the zero level set of $f$. This proves that given any set of dividing curves on a surface, there exists an $\R$-invariant contact form whose characteristic foliation is divided by $\Gamma$ (see Definition \ref{def:div}). Without prescribing a characteristic foliation, an $\mathbb{R}$-invariant contact form inducing a given set of dividing curves was constructed by Lisi \cite{Li}. As we observe in Remark \ref{rem:Alter}, one can alternatively prove Theorem \ref{thm:1} using classical techniques in convex surface theory. Instead, we prove Theorem \ref{thm:1} as a direct consequence of an explicit construction, using local models, of a vector field and an area-form satisfying Theorem \ref{thm:main}.

\begin{theorem}\label{thm:main}
 Let $f:\Sigma \to \R$ be a Morse function on an orientable closed surface $\Sigma$ such that zero is a regular value, and such that it does not admit positive minima nor negative maxima. Then there exists a weakly negative gradient-like vector field $X$ and an area form $\omega \in \Omega^2(\Sigma)$ such that $\pm (\dive_\omega(X))_x>0$ when $\pm f(x)>0$. 
\end{theorem}

For a precise definition of weakly gradient-like vector field, see Definition \ref{def:weakgrad}.

\begin{proof}[Proof of the Theorem \ref{thm:1}]
 We apply Theorem \ref{thm:main} and define a $1$-form by  $\alpha=fdt+\iota_X \omega$. A straightforward computation yields $\alpha \wedge d\alpha =dt \wedge (f\dive_\omega X-X(f))\omega$. By construction $f\dive_\omega X \geq 0$, and the fact the $X$ is weakly gradient-like implies that $-X(f)\geq 0$. Furthermore, they do not vanish simultaneously since the first term only vanishes along the zero level set which is regular by assumption, and $-X(f)$ only vanishes at singular points of $f$. This shows that $\alpha\wedge d\alpha >0$, hence $\alpha$ defines a contact form.
\end{proof}

\begin{remark}
The connection between the nodal sets of eigenfunctions of the Laplacian and dividing sets of convex surfaces of contact $3$-manifolds was first studied by Komendarczyk \cite{Ko}, followed by Lisi \cite{Li}, and is thus a motivation for being able to prescribe the $\mathbb{R}$-invariant function in the expression of an $\mathbb{R}$-invariant contact form. Applying Theorem \ref{thm:1}, we find an $\mathbb{R}$-invariant contact form $\alpha=fdt+\beta$. If there exists a contact metric such that $\alpha=\lambda \star d\alpha$ for a constant $\lambda \neq 0$, and which makes the vector field $\tfrac{\partial}{\partial t}$ orthonormal to $\Sigma\times \{t\}$ near the zero section, then by \cite[Lemma 2.7]{Ko} the Morse function $f$ is an eigenfunction of the Laplacian for some metric on the surface. The existence of such a metric, however, is not clear in general, see also Problem 4.1 in \cite{Ko}.
\end{remark}

In the last section, we use Theorem \ref{thm:1} to compute the degree of the Gauss map associated with an $\mathbb{R}$-invariant contact structure in $\Sigma \times \mathbb{R}$. This gives an alternative proof of the homotopical classification of $\mathbb{R}$-invariant contact structures in terms of their dividing set.

\subsection*{Acknowledgments}
The authors would like to thank John Etnyre, Marc Kegel and Daniel Peralta-Salas for useful comments, as well as Eva Miranda for her constant encouragement. 

\section{Preliminaries}

We will always denote a contact three-manifold by $(M,\xi)$ and assume that the contact structure $\xi$ is coorientable, so $\xi =\ker\alpha$ for a globally defined one-form $\alpha\in \Omega^1(M)$ that satisfies that $\alpha \wedge d\alpha\neq 0$.

We introduce in this section the main definitions, which can all be found in the classical reference \cite{Ge}.
Convex surfaces are a powerful tool to study $3$-dimensional contact manifolds and were introduced in \cite{G1}, other standard references are \cite{Et} and \cite{Mas}. The higher dimensional study of contact convexity has recently been initiated in \cite{HH}.

Given a surface embedded in a three-dimensional contact manifold $(M,\xi)$, the contact structure induces a singular foliation on $\Sigma$, the \emph{characteristic foliation}, which is defined as $\Sigma_\xi=T\Sigma\cap \xi$. The singular points of the foliation are exactly those points $x\in \Sigma$ where $\xi$ is tangent to $T\Sigma$. By the contact condition, the singular points of the characteristic foliation are isolated. Let $\Omega\in \Omega^2(\Sigma)$ be an area form of $\Sigma$. The characteristic foliation can alternatively be defined as an equivalence class of vector fields $[X]$, where $X\in \Omega^1(M)$ is defined by $\iota_X\Omega=\beta$, where $\beta=i^*\alpha$ is the pullback of $\alpha$ by the inclusion map $i:\Sigma\hookrightarrow M$. Two singular vector fields are in the same equivalence class if they differ up to multiplication of a positive function.

A vector field $Y$ is a \emph{contact vector field} if $Y$ preserves the contact structure, that is $\mathcal{L}_Y\alpha=g\alpha$, where $g\in C^\infty(M)$. An embedded surface $\Sigma \subset (M,\xi)$ in a contact manifold is \emph{convex} if there exists, in a neighborhood of $\Sigma$, a contact vector field that is transverse to $\Sigma$. The set $\Gamma \subset \Sigma$ defined by $\Gamma:= \{x \in \Sigma\enspace | \enspace Y_x \in \xi \}$ is called the \emph{dividing set} of the convex contact surface. Equivalently, the $\R$-invariant contact form that defines the contact structure in a neighborhood around the contact surface $\Sigma$ can be written as $\alpha=udt+\beta$, where $u\in C^\infty(\Sigma)$, $\beta \in \Omega^1(\Sigma)$ and $Y=\frac{\partial}{\partial t}$,  and the dividing set is given by $\Gamma= \{x\in \Sigma \enspace|\enspace u(x)=0 \}$.
It follows from the contact condition that $\Gamma$ is an embedded curve and that $\Sigma \setminus \Gamma$ is equipped with an exact symplectic form. 

The dividing set contains all the information about the contact structure in a neighborhood around the convex hypersurface. More precisely, given any two contact structures whose characteristic foliations are divided by the same dividing set up to isotopy, then the contact structures are contact isotopic.

One can show that the dividing set \emph{divides} the characteristic foliation, which means the following.

\begin{definition}\label{def:div}[Definition 4.8.3 in \cite{Ge}]
Let $\Gamma$ be a set of dividing curves that gives a decomposition $\Sigma\setminus \Gamma=\Sigma_+\sqcup \Sigma_-$. A singular $1$-dimensional foliation $\mathcal{F}$ on a closed surface $\Sigma$ is divided by a set of dividing curves $\Gamma$ (or $\Gamma$ divides $\mathcal{F}$) if the following holds:
\begin{itemize}
\item $\mathcal{F}$ is transverse to $\Gamma$,
\item There is an area form $\omega$ and a vector field $X$ defining $\mathcal{F}$ such that $\Sigma_{\pm}=\{p\in \Sigma \enspace | \enspace \pm \operatorname{div}_\omega(X)>0\}$ and $X$ points out of $\Sigma_+$ along $\Gamma$.
\end{itemize}
\end{definition}

\section{Weakly gradient-like characteristic foliations}

In this section, we prove Theorem \ref{thm:main}, i.e. we show that for each Morse function cutting the zero section transversely and without positive minima nor without negative maxima there exists a (negative) weakly gradient-like vector field divided (in the sense of Definition \ref{def:div}) by the zero level set. Recall the definition of weakly gradient-like vector field \cite[Section 9.3]{CE}.

\begin{definition}\label{def:weakgrad}
Let $M$ be a manifold and $f:M \longrightarrow \mathbb{R}$ a smooth function defined on $M$. A vector field $X$ in $M$ is weakly gradient-like for a function $f$ if 
\begin{itemize}
\item $X(p)=0$ if and only if $df(p)=0$,
\item $X(f)>0$ outside the zeroes of $X$.
\end{itemize}
\end{definition}
A \textit{negative} weakly gradient-like vector field satisfies instead $X(f)<0$ outside the zeroes of $X$.  In some parts of the literature, the definition above defines a gradient-like vector field. Here we make a distinction, since a very common definition of gradient-like vector field imposes further that $X$ is exactly the gradient of $f$ with respect to the standard metric on a small enough Morse neighborhood of the critical points of $f$.

As mentioned in the introduction, the conditions imposed on the Morse function are necessary. If there was a positive minimum, a negative gradient-like vector field near that minimum is modeled by a sink and hence cannot have positive divergence. Similarly, at a negative maximum we would have a vector field with a source, which cannot have negative divergence.

\subsection{Decomposing $\Sigma$ into $2$-atoms}

Let $f:\Sigma \longrightarrow \mathbb{R}$ be a Morse function defined on a closed surface $\Sigma$ and satisfying both required properties: it cuts transversely the zero section and does not have positive minimums or negative maximums. Following \cite[Chapter 2]{fomenko}, we define the notion of $2$-atom.

\begin{definition}
A $2$-atom $P$ is a neighborhood of a critical fiber foliated by the level sets of $f$, considered up to fiber equivalence. 
\end{definition} 
That is, given a critical value $c$ of the function, a $2$-atom associated to the critical fiber is given by a connected component $P$ of $f^{-1}[c-\varepsilon,c+\varepsilon]\subset \Sigma$, for $\varepsilon>0$ sufficiently small so that there are no other critical values in $[c-\varepsilon,c+\varepsilon]$. Hence, a $2$-atom is a connected compact surface with boundary,  equipped with a foliation of the level-sets of $f$ and defined up fiber equivalence. Each connected component of the $2$-atom is determined by a graph, whose vertices are critical points and whose edges are the regular part of the critical level set, satisfying:
\begin{itemize}
\item $K=f^{-1}(c) \cap P$ is either a single point or a graph with all vertices of degree $4$,
\item each connected component of $P\setminus K$ is homeomorphic to $S^1\times (0,1]$, and each annulus is either positive or negative (according to the sign of $f-c$) in a way that each edge of $K$ is in the adherence of exactly one positive and one negative annulus.
\end{itemize}

When $K$ is a single point, the $2$-atom corresponds to a neighborhood of a maximum or minimum of the Morse function. Otherwise, each vertex of $K$ represents a hyperbolic critical point. Given the finite set of $2$-atoms of the function $f$, the whole manifold and foliation by level sets is recovered by gluing annuli(foliated by regular circles of $f$) along the corresponding boundary components of the $2$-atoms.\\

In our context, as the Morse function is transverse to the zero-section, zero is a regular value. We will say that a $2$-atom is \textit{positive} if $c>0$, and \textit{negative} if $c<0$. In the subsequent subsections, we will show that there exists a weakly gradient-like flow of positive (respectively negative) divergence on each positive (respectively negative) atom. We restrict to positive atoms since the argument is analogous for the negative ones.

\subsection{Local models in parts of $2$-atoms}

A positive $2$-atom $P$ can be broken into a finite number of  domains $D_1,...,D_n$ and $B_1,...,B_m$, where: 
\begin{itemize}
\item Each $D_i$ is homeomorphic to a closed disk, and is described in a Morse chart neighborhood of one of the critical points in $f^{-1}(c)$. For $\varepsilon$ small enough, the level sets $f^{-1}(\pm \varepsilon)$ intersect the Morse chart and we can choose $D_i$ such that its boundary is given by four arcs lying on the level sets $f^{-1}(c\pm\epsilon)$ and four straight segments (see Figure \ref{fig:model3}).
\item Each $B_j$ is diffeomorphic to $ [0,1]^2$, and satisfies $\overline{P\setminus \bigsqcup_{i=1}^n D_i}=\bigsqcup_{j=1}^m B_j$.
\end{itemize}
In terms of the graph $K$, the disks $D_i$ are neighborhoods of the vertices and the bands $B_j$ are neighborhoods of the edges in $P\setminus \sqcup D_i$.

\begin{figure}[!h]
\begin{tikzpicture}[scale=1]
\draw [black] (-1,2) -- (1,2);
\draw [black] (2,1) -- (2,-1);
\draw [black] (-2,1) -- (-2,-1);
\draw [black] (-1,-2) -- (1,-2);
\draw[blue] (-2,1) arc (270:360:1);
\draw[teal] (1,2) arc (180:270:1);
\draw[blue] (2,-1) arc (90:180:1);
\draw[teal] (-1,-2) arc (0:90:1);
 \draw[teal] (2,1) -- (2.3,1);
\draw[teal] (1,2) -- (1,2.3);
 \node[scale=0.9,teal] at (2,2.2) {$f^{-1}(c+\epsilon)$};
  \node[scale=0.9,blue] at (-2,2.2) {$f^{-1}(c-\epsilon)$};
 \draw[teal] (-2,-1) -- (-2.3,-1);
 \draw[teal] (-1,-2) -- (-1,-2.3);
  \draw[blue] (-2,1) -- (-2.3,1);
  \draw[blue] (-1,2) -- (-1,2.3);
\draw[blue] (2,-1) -- (2.3,-1);
  \draw[blue] (1,-2) -- (1,-2.3);
\draw[-stealth] (-2.3,0) -- (2.3,0);
\draw[-stealth] (0,-2.3) -- (0,2.3);
\end{tikzpicture}

  \caption{The domain $D_i$.} 
  \label{fig:domain}
\end{figure}
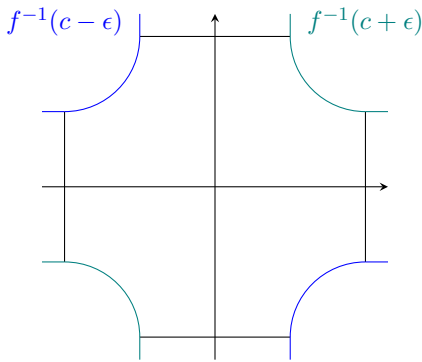

 The boundary of each band $B_j\cong [0,1]^2$ is composed of two segments that are each glued to the boundary segment of some disk, and two segments which are respectively in $f^{-1}(c-\varepsilon)$ and $f^{-1}(c+\varepsilon)$. \\

Let $D_i$ be one of the topological disks of a $2$-atom where $f$ has a critical point. We will define on $D_i$ a vector field and an area form.\\

\paragraph{\textbf{Local model near elliptic point.}}
If $D_i$ is a neighborhood of an elliptic point, the function necessarily has a maximum (since we are in a positive atom) and, by the Morse lemma, there are polar coordinates $(r,\theta)$ on a disk-like neighborhood of the critical point where $f=c-r^2$. In this case, we simply choose the negative gradient of $f$ with respect to the standard metric
$$ X= 2r\pp{}{r}.$$
For the standard area form $\omega=rdr\wedge d\theta$ the divergence is
$$ \operatorname{div}_\omega (X)= 4, $$
which is positive everywhere. \\

\paragraph{\textbf{Local model near hyperbolic points.}} If $D_i$ is a neighborhood of a hyperbolic point, we proceed as follows. The disk $D_i$ lies inside a Morse neighborhood $U$ of $f$ with coordinates $(x,y)$ where $f=c+(x+y)^2-(x-y)^2$. Up to redefining $D_i$ and the bands $B_j$, we can assume that the segments in the boundary of $D_i$ are parallel to $x=\pm \delta$ and $y=\pm \delta$. Consider the vector field 
\begin{align*}
 X&= -(x+y)\Big(\pp{}{x}+\pp{}{y}\Big)+2(x-y)\Big(\pp{}{x}-\pp{}{y}\Big) \\
 &= (x-3y)\pp{}{x} + (y-3x)\pp{}{y}
\end{align*}
which is clearly negative weakly gradient-like flow for $f$. Let $\omega=dx\wedge dy$ be the standard area form, and $X$ satisfies
$$ \operatorname{div}_\omega(X)= 2>0. $$
\begin{figure}[!h]
\begin{tikzpicture}[scale=1.5]
\draw [black] (-1,2) -- (1,2);
\draw [black] (2,1) -- (2,-1);
\draw [black] (-2,1) -- (-2,-1);
\draw [black] (-1,-2) -- (1,-2);
\draw[blue] (-2,1) arc (270:360:1);
\draw[teal] (1,2) arc (180:270:1);
\draw[blue] (2,-1) arc (90:180:1);
\draw[teal] (-1,-2) arc (0:90:1);
 \draw[teal] (2,1) -- (2.3,1);
\draw[teal] (1,2) -- (1,2.3);
 \node[scale=0.8,teal] at (1.6,2.3) {$f^{-1}(c+\epsilon)$};
  \node[scale=0.8,blue] at (-1.6,2.3) {$f^{-1}(c-\epsilon)$};
 \draw[teal] (-2,-1) -- (-2.3,-1);
 \draw[teal] (-1,-2) -- (-1,-2.3);
  \draw[blue] (-2,1) -- (-2.3,1);
  \draw[blue] (-1,2) -- (-1,2.3);
\draw[blue] (2,-1) -- (2.3,-1);
  \draw[blue] (1,-2) -- (1,-2.3);
\draw[gray] (-2.3,-2.3) -- (2.3,2.3);
\draw[gray] (-2.3,2.3) -- (2.3,-2.3);
\draw[-stealth] (-2.3,0) -- (2.3,0);
\draw[-stealth] (0,-2.3) -- (0,2.3);
\foreach \x in {0, ..., 3}{
\draw[gray, rotate=\x*90] (2,1.7) ..controls +(-1.25,-1.25) and +(-1.25,1.25).. (2,-1.7);
}
\draw[orange] (1.7,-1.06) -- (1.7,1.06);
\fill [shift={(1.7,0.0)},scale=0.08,rotate=-90,orange]   (0,0) -- (-1,-0.7) -- (-1,0.7) -- cycle; 
\draw[orange] (1.85,-1.03) -- (1.85,1.03);
\fill [shift={(1.85,0.0)},scale=0.08,rotate=-90,orange]   (0,0) -- (-1,-0.7) -- (-1,0.7) -- cycle; 
\draw[orange] (-1.7,-1.06) -- (-1.7,1.06);
\fill [shift={(-1.7,0.0)},scale=0.08,rotate=90,orange]   (0,0) -- (-1,-0.7) -- (-1,0.7) -- cycle; 
\draw[orange] (-1.85,-1.03) -- (-1.85,1.03);
\fill [shift={(-1.85,0.0)},scale=0.08,rotate=90,orange]   (0,0) -- (-1,-0.7) -- (-1,0.7) -- cycle; 
\draw[orange] (-1.06,1.7) -- (1.06,1.7);
\fill [shift={(0,1.7)},scale=0.08,rotate=180,orange]   (0,0) -- (-1,-0.7) -- (-1,0.7) -- cycle; 
\draw[orange] (-1.03,1.85) -- (1.03,1.85);
\fill [shift={(0,1.85)},scale=0.08,rotate=180,orange]   (0,0) -- (-1,-0.7) -- (-1,0.7) -- cycle; 
\draw[orange] (-1.06,-1.7) -- (1.06,-1.7);
\fill [shift={(0,-1.7)},scale=0.08,rotate=0,orange]   (0,0) -- (-1,-0.7) -- (-1,0.7) -- cycle; 
\draw[orange] (-1.03,-1.85) -- (1.03,-1.85);
\fill [shift={(0,-1.85)},scale=0.08,rotate=0,orange]   (0,0) -- (-1,-0.7) -- (-1,0.7) -- cycle; 
\fill [shift={(-0.7,0.7)},scale=0.08,rotate=135,gray]   (0,0) -- (-1,-0.7) -- (-1,0.7) -- cycle; 
\fill [shift={(-0.7,-0.7)},scale=0.08,rotate=45,gray]   (0,0) -- (-1,-0.7) -- (-1,0.7) -- cycle; 
\fill [shift={(0.7,-0.7)},scale=0.08,rotate=-45,gray]   (0,0) -- (-1,-0.7) -- (-1,0.7) -- cycle; 
\fill [shift={(0.7,0.7)},scale=0.08,rotate=-135,gray]   (0,0) -- (-1,-0.7) -- (-1,0.7) -- cycle; 
\fill [shift={(0,-1.07)},scale=0.08,rotate=0,gray]   (0,0) -- (-1,-0.7) -- (-1,0.7) -- cycle; 
\fill [shift={(1.07,0)},scale=0.08,rotate=-90,gray]   (0,0) -- (-1,-0.7) -- (-1,0.7) -- cycle; 
\fill [shift={(0,1.07)},scale=0.08,rotate=180,gray]   (0,0) -- (-1,-0.7) -- (-1,0.7) -- cycle; 
\fill [shift={(-1.07,0)},scale=0.08,rotate=90,gray]   (0,0) -- (-1,-0.7) -- (-1,0.7) -- cycle; 
\end{tikzpicture}

  \caption{The model around a hyperbolic point. In green and blue the level-set $f^{-1}(c\pm\epsilon)$ and in orange the cut-off region where the vector field is parallel to the straight boundary segment.} 
  \label{fig:model3}
\end{figure}
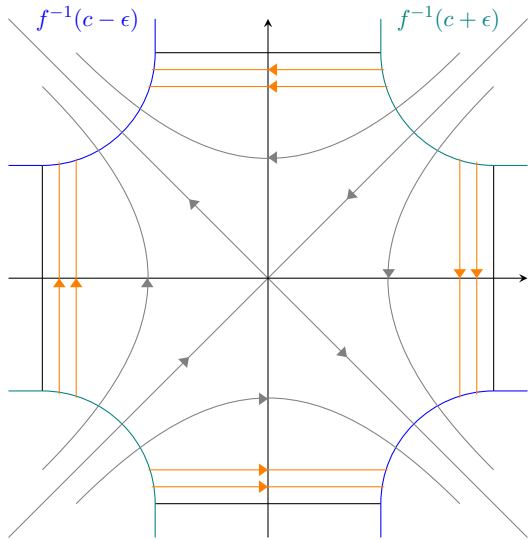

We want to modify $X$ near $x=\pm \delta$ and $y=\pm \delta$ so that it is parallel respectively to the $y$ and $x$ axis there. We will explain how to do it near $x= \pm \delta$ and in the $y$ coordinate it is done analogously. Choose $\delta_1, \delta_2$ positive such that $\delta_1<\delta_2<\delta$, and take the region $V_+=\{x\in [\delta_1,\delta] \}\cap D_i$. There, the vector field is of the form
\begin{equation}\label{eq:X}
 X= g \pp{}{x} + h \pp{}{y},
\end{equation}
with $h<0$ in $V_+$ and $\pp{g}{x}+ \pp{h}{y}>0$ in all $D_i$. Let $\varphi_1, \varphi_2 \in C^\infty(\R)$ denote cut-off functions such that
\begin{equation*}
\begin{cases}
\varphi_1(x)&=0 \text{ for } x\leq \delta_1,\\
\varphi_1(x)&=1 \text{ for } x\geq \delta_2,\\
\varphi_2(x)&=1 \text{ for } x\leq \delta_2,\\
\varphi_2(x)&=0 \text{ for } x\geq \delta-\delta' \text{ for some very small } \delta'>0.
\end{cases} 
\end{equation*}

\begin{figure}[h!]
\begin{tikzpicture}[scale=2]
\draw[->] (-0.5,0) -- (3,0);
\draw[->] (0,-0.3) -- (0,1.2);
\draw[blue] (-0.5,0.01) -- (0.5,0.01);
\draw[blue] (0.5,0.01) ..controls +(0.5,0) and + (-0.5,0).. (1.5,1.005);
\draw[blue] (1.5,1.005) -- (2.8,1.005);
 \node[scale=1,blue] at (2.7,0.8) {$\varphi_1$};
\draw[teal] (-0.5,1) -- (1.6,1);
\draw[teal] (1.6,1) ..controls +(0.5,0) and + (-0.5,0).. (2.6,0.01);
\draw[teal] (2.6,0.01) -- (2.8,0.01);
 \node[scale=1,teal] at (2.7,0.2) {$\varphi_2$};
\draw (0.4,-0.05) -- (0.4,0.05);
 \node[scale=1] at (0.4,-0.2) {$\delta_1$};
\draw (1.55,-0.05) -- (1.55,0.05);
 \node[scale=1] at (1.55,-0.2) {$\delta_2$};
\draw (2.7,-0.05) -- (2.7,0.05);
 \node[scale=1] at (2.7,-0.2) {$\delta$};
\end{tikzpicture}
\caption{The functions $\varphi_1$ and $\varphi_2$.}
\end{figure}
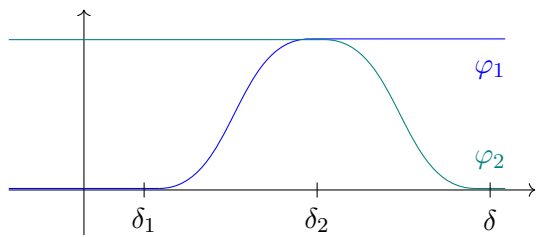
 We choose a function $u(y)$ that is negative for $y\in [-\delta,\delta]$ and such that $u'(y)\gg\varphi_2'(x)$. In $V_+$, we modify the vector field to
$$ Y= \varphi_2(x)g\pp{}{x} + \Big(h+\varphi_1(x) u(y)\Big) \pp{}{y},$$
which glues well with the previously defined vector field near $x=\delta_1$.  It still satisfies that the $\pp{}{y}$ component is negative, and the $\pp{}{x}$ component is zero near $x=\delta$. Furthermore, the divergence is given by
\begin{equation*}
\operatorname{div}_\omega(Y)= 
\begin{cases}
\pp{g}{x} + \pp{h}{y}+ \varphi_1(x) \cdot u'(y) >0, \enspace \text{for } x\in [\delta_1,\delta_2],\\
\varphi_2(x) \pp{g}{x} + \varphi_2'(x)g + \pp{h}{y}+u'(y)>0, \enspace \text{for } x\in [\delta_2,\delta],
\end{cases}
\end{equation*}
where in the last equation we used that $u'(y)$ is positive and can be taken arbitrarily large. We emphasize that for $x$ close to $\delta$, $Y=(h+u)\frac{\partial}{\partial y}$ and $\dive_\omega Y=\frac{\partial (h+u)}{\partial y} >0$.\\

We can now proceed analogously in $V_-=\{x\in [-\delta, -\delta_1]\} \cap D_i$. We keep denoting $X$ in the form \eqref{eq:X}, where now $h>0$ in $V_-$. We choose cutoff functions $\phi_1$ and $\phi_2$ such that
 
 \begin{equation*}
 \begin{cases}
\phi_1(x)&=0 \text{ for } x\geq -\delta_1,\\
\phi_1(x)&=1 \text{ for } x\leq -\delta_2,\\
\phi_2(x)&=1 \text{ for } x\geq -\delta_2,\\
\phi_2(x)&=0 \text{ for } x\leq -\delta+\delta' \text{ for some very small } \delta'>0.
\end{cases} 
\end{equation*}
Let $v(y)$ be a positive function with some arbitrarily large positive derivative in $y\in[-\delta,\delta]$. We change the vector field to
$$Y'= \phi_2(x)g\pp{}{x} + \Big(h+\phi_1(x) v(y)\Big)\pp{}{y}.$$
This vector field glues well with $X$, it always have a positive $\pp{}{y}$-component and near $x=\delta$ it is parallel to $\pp{}{y}$. Furthermore, the divergence of $Y'$ satisfies
\begin{equation*}
\operatorname{div}_\omega(Y')= 
\begin{cases}
\pp{g}{x}+ \pp{h}{y}+ \phi_1(x)\cdot v'(y) >0, \enspace \text{for } x\in[-\delta_2,-\delta_1],\\
\phi_2(x) \pp{g}{x} + \phi_2'(x)g + \pp{h}{y}+v'(y)>0, \enspace \text{for } x\in[-\delta,-\delta_2].
\end{cases}
\end{equation*}
where we used in the second equation that $v'(y$) can be taken arbitrarily large. As before, we emphasize that for $x$ close to $-\delta$, $Y=(h+v)\frac{\partial}{\partial y}$ and $\dive_\omega Y=\frac{\partial (h+v)}{\partial y} >0$.\\

Arguing analogously near $y=\pm\delta$, we construct a vector field $X$ in $D_i$ that is negative weakly gradient-like for the Morse function, has everywhere positive divergence, and that is parallel to the boundary segments that will be glued to the bands $B_j$.

\begin{remark}
If the $2$-atom is negative, we want a vector field that has negative divergence everywhere. The same model can be constructed, taking first a vector field of the form
\begin{align*}
 X&= -2(x+y)\Big(\pp{}{x}+\pp{}{y}\Big)+(x-y)\Big(\pp{}{x}-\pp{}{y}\Big) \\
  &=(-3y-x)\pp{}{x} + (-3x-y)\pp{}{y} ,
 \end{align*}
which is again negative weakly gradient-like for the Morse function but has divergence equal to $-2$ instead (with respect to the standard area form in the disk). The whole construction is then analogous, but the functions $u(y)$ and $v(y)$ are chosen with large negative derivative instead.
\end{remark}

\subsection{Characteristic foliation on each atom}

We will now glue the vector fields and area-forms that we constructed near the hyperbolic points adapting a standard trick \cite[page 231]{Ge}. Every boundary segment of $\bigsqcup_{i=1}^n D_i$ is connected to another boundary segment through some band $B_j$. Let $B\cong [0,1]^2$ be one of the bands endowed with coordinates $(t,z)$. We can assume that the level sets of $f$ correspond to the level sets of $z$, and that $f>c$ in $z>0$ and $f<c$ in $z<0$. We might extend the vector field and area-form constructed in each attached disk to obtain a pair $(X_0, \omega_0)$ and $(X_1, \omega_1)$ defined respectively near $t=0$ and near $t=1$. The vector fields $X_0$ and $X_1$ are of the form 
$$ X_0|_{\{t=0\}}= g_0 \pp{}{z}, \qquad X_1|_{\{t=1\}}=g_1 \pp{}{z}, $$
where $g_0$ and $g_1$ are both negative functions. Furthermore we have 
$$ \operatorname{div}_{\omega_0}(X_0)>0, \qquad \operatorname{div}_{\omega_1}(X_1)>0. $$

 We might assume that $dt\wedge dz$ induces the same orientation as $\omega_0$ and $\omega_1$. We first interpolate the area form $\omega_0$ to the constant area form $dt \wedge dz$ close to $\{t=0\}$ while having the condition of positive divergence of $X_0$ being satisfied. We consider a positive cut-off function $\phi$ such that $\phi(x)=1$ close to $x=0$ and $\phi(x)=0$ for $x$ close to $\delta$ (for some small $\delta>0$). We consider the area form given by $\omega= (c_0\phi(t)+(1-\phi(t)) dt\wedge dz$. The divergence of $X_0$ with respect to $\omega$ is given by a positive multiple of 
 $$\phi (t) \Big(\frac{\partial c_0}{\partial z}g_0+c_0 \frac{\partial g_0}{\partial z}\Big)+\Big(1-\phi(t)\Big)\frac{\partial g_0}{\partial z}.$$
 The first term is given by a positive multiple of $\dive_{\omega_0} X_0$ and hence is positive. The second term is also positive. Indeed, we can assume that the coordinate $z$ coincides with the coordinate defining the boundary segment on any of the disks $D_i$ that lie in the boundary of the band. By construction, the component of $X$ with respect to that coordinate has a positive derivative, as we emphasized in the previous section, and thus $\pp{g_0}{z}>0$ near $t=0$. The same construction can be done for $(X_1,\omega_1)$ close to $\{t=1\}$.
We thus obtain a globally defined area form $\omega$ and we take a vector field $$X=g\frac{\partial}{\partial z},$$ with $\frac{\partial g}{\partial z}>0$  and $g=g_0$ and $g=g_1$ respectively near $t=0$ and $t=1$. The vector field $X$ is a negative weakly gradient-like vector field for the Morse function, and for $t\in(\delta,1-\delta)$ it further satisfies
$$ \operatorname{div}_\omega(X)=\pp{g}{z} >0.$$

Doing this construction along each band, we construct a negative weakly gradient flow $X$ and an area form $\omega$ on the positive $2$-atom $P$ such that $\operatorname{div}_\omega(X)>0$ everywhere.

\subsection{Conclusion}

We have now constructed a vector field and an area form such that the vector field has positive (respectively negative) divergence in each positive (respectively negative) $2$-atom of the Morse function. {It only remains to extend this vector field and area form along the annuli, foliated by regular level sets of $f$, whose boundary are boundary components of $2$-atoms.  We will adapt an argument from \cite[p. 231]{Ge} to obtain an area form $\omega$ and vector field $X$ that interpolate between the ones that are given on the boundary of the annulus so that $\operatorname{div}_\omega(X)=0$ only at the zero level set.

 When we glue together two positive atoms or two negative atoms, we can adapt the argument as follows. Assume that the annulus $S^1 \times [-1,1]$ has both boundary components on positive $2$-atoms. Let $X_1$ be the weakly gradient like vector field and $\Omega_1$ the area form defined near $S^1\times \{-1\}$, and $X_2$ and $\Omega_2$ the vector field and area form defined near $S^1\times \{1\}$. We can assume that the coordinates $(\theta,s)$ of $S^1\times [-1,1]$ are such that $X_1=-\frac{\partial}{\partial s}$  near $t=-1$ and $X_2=-\frac{\partial}{\partial s}$ near $t=1$. We will also assume that $d\theta \wedge ds$ induces the same orientation as $\Omega_1$ and $\Omega_2$. Let $\Omega$ be an area form such that $\Omega= \Omega_1$ for $s\in (\epsilon,1)$ and $\Omega=\Omega_2$ for $s\in (-1,-\epsilon)$, $\epsilon>0$ small. The divergence of $X=\pp{}{s}$ with respect to $g(s)\Omega$ where $g\in C^\infty((-1,1),\R)$ is given by 
$$\dive_{g\Omega}X=\frac{-g'(s)}{g(s)}+\dive_\Omega(X).$$
As $\dive_\Omega(X)$ is bounded, to obtain positive divergence it is sufficient to take a function $g$ with a large negative derivative. We thus take the function $g$ to be decreasing such that $g(s)=K$ for $s\in (-1,\epsilon)$ where $K$ is a big constant and $g(s)=1$ when $s\in (\epsilon,1)$. Note that the area form $g\Omega$ only glues with $\Omega_1$ up to multiplication of a constant. This is not a problem in our construction as the area form constructed in the $2$-atom can be changed by multiplying it by the same constant, and the vector field $X_1$ will still have everywhere positive divergence. The same construction holds when gluing two negative $2$-atoms. 

When we are gluing a positive atom with a negative one (this implies that there is a circle in the annulus that corresponds to a component of the zero level set of the function), we are exactly in the case treated in \cite[p 231]{Ge}, that we sketch for completeness. Keeping the notation as above, assume that the circle in the zero level set of the function is given by $s=0$, choose some area-form $\Omega'$ in $s\in (-\varepsilon/2,\varepsilon/2)$ for $\epsilon>0$ small enough, for which $\operatorname{div}(X)$ vanishes exactly at $s=0$ and is positive or negative respectively for $s>0$ and $s<0$. We might now choose any area-form $\Omega$ that coincides with $\Omega_1,\Omega', \Omega_2$ respectively for $s\in [-1,-\varepsilon]$, $s\in [-\varepsilon/2, \varepsilon/2]$ and $s\in [\varepsilon,1]$.  Up to choosing $\Omega$ that coincides with suitable positive constant multiples of $\Omega_1$, $\Omega'$ and $\Omega_2$, we construct an area-form $g(s)\Omega$, where $g(s)$ is a positive function that is constant for $|s| \not \in [\varepsilon/2,\varepsilon]$, and either decreases (respectively increases) fast enough for $s\in (-\varepsilon, -\varepsilon/2)$ (respectively for $s\in (\varepsilon/2,\varepsilon)$).

After doing this at each boundary circles of the domains $S_i$, we obtain a globally defined vector field $X$ in $\Sigma$, endowed with an area form $\omega$ such that $X$ is divided by the original dividing set (as in Definition \ref{def:div}), and $X$ is weakly gradient-like for the Morse function $f$. This finishes the proof of Theorem \ref{thm:main}.

\begin{remark}\label{rem:Alter}
As we mentioned in the introduction, one can use classical results in convex surface theory to prove Theorem \ref{thm:1}, we sketch here the argument for completeness. Let $f$ be a Morse function satisfying the hypothesis of the theorem. Choose any Riemannian metric $g$ in $\Sigma$ such that $X=-\operatorname{grad}_g f$ has either positive (when $f$ is positive there) or negative (when $f$ is negative there) divergence with respect to some area-form at the singular points of $f$. This is a sufficient condition for the existence of a (not necessarily $\mathbb{R}$-invariant) contact form $\alpha$ in $\Sigma \times \mathbb{R}$ such that the characteristic foliation along $\Sigma \times \{0\}$ is $X$ (see Lemma 4.6.3 in \cite{Ge}). Since $X$ is in particular of ``almost" Morse-Smale type, the proof of \cite[Section II, Proposition 2.6]{G1} (see also Proposition 4.8.7 and Remark 4.8.9 in \cite{Ge}) shows that $X$ is divided by some dividing set. An inspection of the proof shows that, in our case, the dividing set can be chosen to be the zero level set of $f$. This proves Theorem \ref{thm:main} and hence also Theorem \ref{thm:1}. 
\end{remark}

\section{Homotopical classification using Morse functions}

Given a closed surface $\Sigma$, a dividing set $\Gamma$, and a sign to each side of $\Gamma$, a germ of $\mathbb{R}$-invariant contact structure $\xi_\Gamma$ in $M=\Sigma \times \mathbb{R}$ is uniquely determined, up to isotopy, by the set of cooriented curves $\Gamma$ \cite{G1}. The homotopy class of a plane field $\xi$ in $M$ is determined by the homotopy class of any vector field $Y$ such that $Y\oplus \xi=TM$. Fixing a trivialization of the unit tangent bundle $TSM\cong M\times S^2$, the homotopy class of $Y$ (and thus also the homotopy class of $\xi$) is given by the degree of the Gauss map induced by $Y$, a map 
$$ f: M \longrightarrow S^2. $$
We call this map the Gauss map associated with $\xi$, as in \cite[p. 137]{Ge}. In this section, we will apply Theorem \ref{thm:1} to a completely explicit height Morse function to compute the degree of the Gauss map associated with $\xi_{\Gamma}$. We will give an explicit embedding 
$$e:\Sigma \longrightarrow \mathbb{R}^3,$$
for which the coordinate $z$ induces a Morse function on $\Sigma$, and it is satisfied that
\begin{align*}
\Sigma\cap \{z=0\}&=\Gamma, \\
\Sigma \cap \{z>0\}&=\Sigma_+, \\
\Sigma \cap \{z<0\}&=\Sigma_-. 
\end{align*}
Abusing notation, we will denote by $\Sigma_+$ and $\Sigma_-$ the closure of the respective open surface. We choose an embedding such that each connected component of $\Sigma_+$ and $\Sigma_-$ is embedded in a ``standard" way for which $z$ induces a Morse function $h$. In particular,  on each connected component of $\Sigma_+$ or $\Sigma_-$, the function $h$ has a single elliptic point (which is either a maximum or a minimum) and its boundary is the zero level set. The zero level set is the minimum or maximum of the height function in that component of $\Sigma_+$ or $\Sigma_-$. Doing this for each connected component of $\Sigma_+$ and $\Sigma_-$, carefully matching the boundary circles, we construct a completely explicit embedding of $\Sigma$, see for example Figure \ref{fig:morse}.

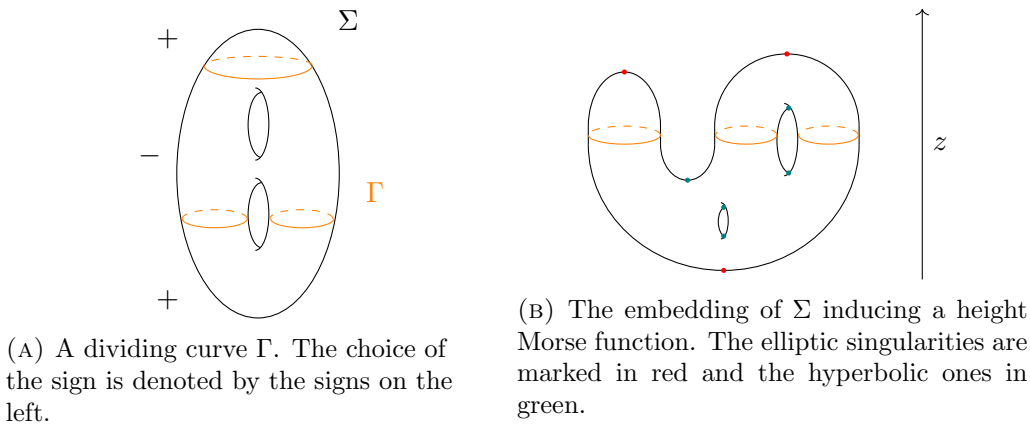
\begin{figure}[!h]
\centering
\begin{subfigure}{.5\textwidth}
  \centering
\begin{tikzpicture}[scale=1.2]
\draw (0,0) ellipse (.9 and 1.6);
 \node at (1,1.7) {$\Sigma$};
 \node at (-1,1.5) {$+$};
 \node at (-1.2,0.2) {$-$};
 \node at (-1,-1.4) {$+$};
\draw (0.04,0.9) arc (90:265:.15 and .36);
\draw (-0.03,0.15) arc (-90:90:.15 and .4);
\draw (0.04,-0.1) arc (90:265:.15 and .36);
\draw (-0.03,-0.85) arc (-90:90:.15 and .4);
 \node[orange] at (1.3,-0.2) {$\Gamma$};
 \draw[orange] (-0.6,1.2) arc (180:360:.6 and .15);
\draw[orange, dashed] (0.6,1.2) arc (0:180:.6 and .1);
\draw[orange] (-0.84,-0.5) arc (180:360:.36 and .1);
\draw[orange,dashed] (-0.12,-0.5) arc (0:180:.36 and .1);
\draw[orange] (0.14,-0.5) arc (180:360:.35 and .1);
\draw[orange, dashed] (0.83,-0.5) arc (0:180:.35 and .1);
\end{tikzpicture}
  \caption{A dividing curve $\Gamma$. {The choice of \\ the sign is denoted by the signs on the\\ left.}}
    \label{fig:divcurv}
\end{subfigure}%
\begin{subfigure}{.5\textwidth}
  \centering
\begin{tikzpicture}[scale=1.2]
\draw[->] (3.7,-1.5)--(3.7,1.5);
\node at (3.9,0) {$z$};

\draw (0,0) arc (-180:0:1.5 and 1.4);
\draw (0.8,0.2) arc (0:180:0.4 and 0.6);
\draw (0,0.2) -- (0,0);
\draw (0.8,0.2) -- (0.8,0);
\draw (3,0.2) arc (0:180:0.8 and 0.8);
\draw (3,0.2) -- (3,0);
\draw (1.4,0.2) -- (1.4,0);
\draw (0.8,0) arc (-180:0:0.3 and 0.4);
\draw (2.24,0.4) arc (90:265:.15 and .36);
\draw (2.17,-0.35) arc (-90:90:.15 and .4);
\draw (1.52,-0.7) arc (90:265:.08 and .16);
\draw (1.47,-1.05) arc (-90:90:.08 and .2);
\draw[orange] (0,0.1) arc (180:360:0.4 and .1);
\draw[orange, dashed] (0.8,0.1) arc (0:180:.4 and .1);
\draw[orange] (2.33,0.1) arc (180:360:0.33 and .1);
\draw[orange, dashed] (2.99,0.1) arc (0:180:.33 and .1);
\draw[orange] (1.41,0.1) arc (180:360:0.34 and .1);
\draw[orange, dashed] (2.09,0.1) arc (0:180:.34 and .1);
\fill[red] (0.4,0.8) circle[radius=0.8pt];
\fill[red] (2.2,1) circle[radius=0.8pt];
\fill[red] (1.5,-1.4) circle[radius=0.8pt];
\fill[teal] (1.1,-0.4) circle[radius=0.8pt];
\fill[teal] (2.22,-0.32) circle[radius=0.8pt];
\fill[teal] (2.22,0.4) circle[radius=0.8pt];
\fill[teal] (1.5,-1.02) circle[radius=0.8pt];
\fill[teal] (1.5,-0.7) circle[radius=0.8pt];
\end{tikzpicture}
  \caption{The embedding of $\Sigma$ inducing a height Morse function. The elliptic singularities are marked in red and the hyperbolic ones in green.}
  \label{fig:embed}
\end{subfigure}
\caption{Morse height function associated to a dividing set.}
\label{fig:morse}
\end{figure}

The embedding of $\Sigma$ into $\mathbb{R}^3$ inducing the height Morse function $h$ fixes a trivialization of $TM|_{\Sigma}$ given by the standard unit sphere in $\mathbb{R}^3$. The contact structure is defined, up to isotopy by the kernel of the contact form
$$ \alpha= h dt + \iota_X\omega, $$
where $t$ is the coordinate in the second factor of $\Sigma \times \mathbb{R}$.
To compute the degree of the Gauss map associated with $\xi_{\Gamma}=\ker \alpha$, we will first choose a section $Y$ of $TM|_{\Sigma\times \{0\}}$ that is everywhere transverse to $\xi_{\Gamma}$ (positively with respect to $\alpha$). Let $A$ be a small open neighborhood of $\Gamma=\{h=0\}$, and {we see $\Sigma$ as an embedded submanifold in $\mathbb{R}^3$ with the embedding inducing the Morse function $h$}. Away from $A$, we can choose $Y|_{\Sigma_+}=\nu$ and $Y|_{\Sigma_-}=-\nu$, where $\nu$ denotes an orthonormal section (with respect to the standard metric in $\mathbb{R}^3$) of $T\Sigma$ pointing outwards the boundary of the compact domain bounded by $\Sigma \subset \mathbb{R}^3$. It remains to extend this section along $TM|_A$. Seeing each component of $A$ as a cylinder in $\mathbb{R}^3$, the section $\nu$ points outwards on one side of a component of $\Gamma$ and inwards on the other side of the component $\Gamma$. Furthermore, since $h=0$ along $\Gamma$ and $\alpha$ is non-vanishing, $\iota_X\omega$ induces an orientation on $\Gamma$. Hence, we extend the section $Y$ along $A$ by rotation around the $z$-axis, such that $Y|_\Sigma$ is tangent to $\Sigma$ and pairs positively with $\beta$. \\

Having defined $Y$, we look at it as a section of $ST\Sigma$, and it defines a map
$$ f: \Sigma \longrightarrow S^2.$$
We denote by $e_+$ (respectively $e_-$) the number of positive (respectively negative) elliptic points and by $h_+$ (respectively $h_-$) the number of positive (respectively negative) hyperbolic points. The genus of $\Sigma_+$ (respectively $\Sigma_-$) is denoted by $g_+$ (respectively $g_-$).

We have the following formulae:

\begin{itemize}
\item $e_\pm= \#\Sigma_\pm$,
\item  $h_\pm= \#\partial \Sigma_\pm - \#\Sigma_\pm + 2 g_\pm$.
\end{itemize}

\begin{lemma}
The degree of the Gauss map $f$ associated with $\xi|_\Gamma$ is given by $$\deg(f)=\frac{1}{2}(\chi(\Sigma_+)-\chi(\Sigma_-)).$$
\end{lemma}

\begin{proof}
 By the local degree theorem, the degree is given by the sum of the local degree over all the preimages of the south pole.  Recall that the plane field $\xi_\Gamma$ is parallel  to $T\Sigma$ at the singular points of the characteristic foliation, hence at the singular points of the Morse function. The orientation of $\xi_\Gamma$ is the same as that of $T\Sigma$ (i.e. it points outwards) at the singular points in $\Sigma_+$, and has the opposed orientation at the singular points in $\Sigma_-$. We observe that for $\Sigma_\pm$ of genus $g$, there exist additionally $2g$-hyperbolic points, but only for half of them the Gauss map points to the south pole.
 
 \begin{figure}[!h]
 \includegraphics[scale=0.3]{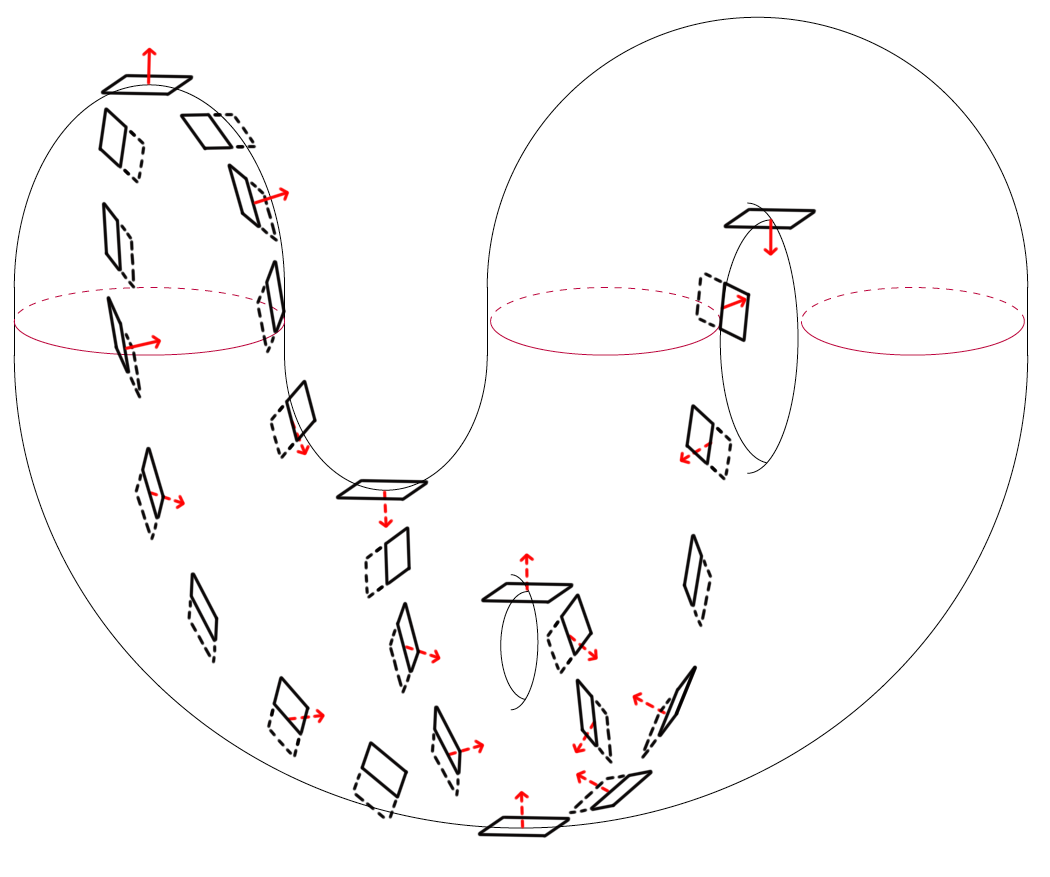}
 \caption{An example of an explicit Morse function and its cooriented contact structure $\xi_\Gamma$. A dashed (non-dashed) red arrow indicates that the vector orthonormal to the plane points inwards (outwards) the surface.}
 \end{figure}
 
  The Gauss map is a local homeomorphism near each of these points, so the local degree is either $1$ or $-1$ if this homeomorphism is orientation-preserving or orientation-reversing respectively. One easily checks that it is orientation-reversing near a positive hyperbolic point, and orientation-preserving near a negative hyperbolic point. Hence, the degree of the map is given by $\deg(f)=h_--h_+-g_-+g_+$. It follows from the description of the Morse function that
$$h_--h_+= -\#\Sigma_-+\#\Sigma_+ + 2g_{-}- 2g_+,$$
where we used that $\#\partial \Sigma_+= \#\partial \Sigma_-$. Hence we deduce that
\begin{align*}
\deg(\xi_{\Gamma})&= h_--h_+-g_-+g_+\\
&=-\#\Sigma_-+\#\Sigma_++g_- - g_+ \\
&= \frac{1}{2} (\chi(\Sigma_+)-\chi(\Sigma_-)),
\end{align*} 
as claimed in the statement.
\end{proof}
In other words, two $\mathbb{R}$-invariant contact structures in $\Sigma \times \mathbb{R}$ defining domains $\Sigma_+, \Sigma_-$ and $\Sigma_+',\Sigma_-'$ (according to their dividing sets $\Gamma$ and $\Gamma'$) are homotopic through plane fields if and only if $ \frac{1}{2} (\chi(\Sigma_+)-\chi(\Sigma_-))=\frac{1}{2} (\chi(\Sigma_+')-\chi(\Sigma_-'))$.
\begin{remark}
Two contact structures are homotopic through plane fields in $\Sigma \times \mathbb{R}$ if and only if they are homotopic through contact structures according to the $h$-principle for contact structures in open manifolds \cite{Gr}.
\end{remark}
Another way of homotopically classifying contact structures in a neighborhood of convex surfaces is computing the Euler class of the plane field (see e.g. \cite{Et, H}). It is given by 
\begin{align*}
e(\xi_{\Gamma})&= \chi(\Sigma_+)-\chi(\Sigma_-),
\end{align*}
that is, by twice the degree of the Gauss map associated with the contact structure.

\bibliographystyle{alpha}

\end{document}